\documentclass[11pt,reqno]{amsart}

\usepackage{amsmath}
\usepackage{amssymb}
\newtheorem{theorem}{Theorem}[section]
\newtheorem{lemma}[theorem]{Lemma}

\newtheorem{proposition}[theorem]{Proposition}

\newtheorem{remark}[theorem]{Remark}

\numberwithin{equation}{section}

\begin{document}
\title[Ideals of the Toeplitz Algebra]{On a Class of Ideals of the Toeplitz Algebra on the Bergman Space of the Unit Ball}
\author{Trieu Le}
\address{Trieu Le, Department of Mathematics, State University of New York at Buffalo, Buffalo, NY 14260, USA}
\email{trieule@buffalo.edu}
\begin{abstract}
Let $\mathfrak{T}$ denote the full Toeplitz algebra on the Bergman space of the unit ball $\mathbb{B}_n.$ For each subset $G$ of $L^{\infty},$ let $\mathfrak{CI}(G)$ denote the closed two-sided ideal of $\mathfrak{T}$ generated by all $T_fT_g-T_gT_f$ with $f,g\in G.$ It is known that $\mathfrak{CI}(C(\overline{\mathbb{B}}_n))=\mathcal{K}$ - the ideal of compact operators and $\mathfrak{CI}(C(\mathbb{B}_n))=\mathfrak{T}.$ Despite these ``extremal cases'', $\mathfrak{T}$ does contain other non-trivial ideals. This paper gives a construction of a class of subsets $G$ of $L^{\infty}$ so that $\mathcal{K}\subsetneq\mathfrak{CI}(G)\subsetneq\mathfrak{T}.$
\end{abstract}
\maketitle

\section{INTRODUCTION}

For any integer $n\geq 1,$ let $\mathbb{C}^{n}$ denote the Cartesian product of $n$ copies of $\mathbb{C}.$ For $z=(z_1,\ldots,z_n)$ and $w=(w_1,\ldots,w_n)$ in $\mathbb{C}^n,$ we write $\langle z,w\rangle=z_1\overline{w}_1+\cdots+z_n\overline{w}_n$ and $|z|=\sqrt{|z_1|^2+\cdots+|z_n|^2}$ for the inner product and the associated Euclidean norm. Let $\mathbb{B}_n$ denote the open unit ball which consists of all $z\in\mathbb{C}^n$ with $|z|<1.$ Let $\mathbb{S}_n$ denote the unit sphere which consists of all $z\in\mathbb{C}^n$ with $|z|=1.$ For any subset $V$ of $\mathbb{B}_n$ we write ${\rm cl}(V)$ to denote the closure of $V$ as a subset of $\mathbb{C}^{n}$ with respect to the usual Euclidean metric. We write $\overline{\mathbb{B}}_n$ for the closed unit ball which is also ${\rm cl}(\mathbb{B}_n).$ Let $C(\mathbb{B}_n)$ (and $C(\overline{\mathbb{B}}_n)$) denote the space of all functions that are continuous in the open unit ball (respectively, the closed unit ball.)

Let $\nu$ denote the Lebesgue measure on $\mathbb{B}_n$ normalized so that $\nu(\mathbb{B}_n)=1$. Let $L^{2}=L^{2}(\mathbb{B}_n,\mathrm{d}\nu)$ and $L^{\infty}=L^{\infty}(\mathbb{B}_n,\mathrm{d}\nu).$ The Bergman space $L^{2}_{a}$ is the subspace of $L^2$ which consists of all analytic functions. The normalized reproducing kernels for $L_{a}^{2}$ are of the form
\begin{equation*}
k_z(w)=(1-|z|^2)^{(n+1)/2}(1-\langle w,z\rangle )^{-n-1},\ |z|,|w|<1.
\end{equation*}
We have $\|k_z\|=1$ and $\langle g,k_z\rangle =(1-|z|^2)^{(n+1)/2}g(z)$ for all $g\in L_{a}^{2}.$

The orthogonal projection from $L^2$ onto $L^2_a$ is given by
\begin{equation*}
(Pg)(z)=\int\limits_{\mathbb{B}_n}\dfrac{g(w)}{(1-\langle z,w\rangle )^{n+1}}\ \mathrm{d}\nu(w),\ g\in L^2,\ z\in \mathbb{B}_n.
\end{equation*}

For any $f\in L^{\infty}$ the Toeplitz operator $T_{f}:L_{a}^{2}\longrightarrow L_{a}^{2}$ is defined by $T_fh=P(fh)$ for $h\in L^2_a.$ We have
\begin{equation}\label{eqn-2}
(T_fh)(z) = \int\limits_{\mathbb{B}_n}\dfrac{f(w)h(w)}{(1-\langle z,w\rangle)^{n+1}}\ {\rm d}\nu(w)
\end{equation}
for $h\in L^2_a$ and $z\in\mathbb{B}_n.$

For all $f\in L^{\infty}, \|T_f\|\leq\|f\|_{\infty}$ and $T_f^{*}=T_{\overline{f}}.$ In contrast with Toeplitz operators on the Hardy space of the unit sphere, there are $f\in L^{\infty}$ so that $\|T_f\|<\|f\|_{\infty}$. Also, since $T_f$ is an integral operator by equation \eqref{eqn-2}, we see that if $f$ vanishes almost everywhere in the complement of a compact subset of $\mathbb{B}_n$ then $T_f$ is compact.

Let $\mathfrak{B}(L_a^2)$ be the C$^{*}-$algebra of all bounded linear operators on $L_a^2$. Let $\mathcal{K}$ denote the ideal of $\mathfrak{B}(L^2_a)$ that consists of all compact operators. The full Toeplitz algebra $\mathfrak{T}$ is the C$^{*}-$subalgebra of $\mathfrak{B}(L^2_a)$ generated by $\{T_f: f\in L^{\infty}\}.$ For any subset $G$ of $L^{\infty},$ let $\mathfrak{I}(G)$ denote the closed two-sided ideal of $\mathfrak{T}$ generated by all $T_f$ with $f\in G.$ Let $\mathfrak{CI}(G)$ denote the closed two-sided ideal of $\mathfrak{T}$ generated by all commutators $[T_f, T_g]=T_fT_g-T_gT_f$ with $f,g\in G.$ A result of Lewis Coburn \cite{Coburn} in the 70's showed that $\mathfrak{CI}(C(\overline{\mathbb{B}}_n))=\mathcal{K}.$ Recently Daniel Su{\'a}rez \cite{Suarez2004} showed that $\mathfrak{CI}(L^{\infty})=\mathfrak{T}$ for the one-dimensional case. This result has been generalized by the author \cite{Le2006} to all $n\geq 1.$ In fact, we are able to show that $\mathfrak{CI}(G)=\mathfrak{T}$ for certain subsets $G$ of $L^{\infty}.$ We can take $G=\{f\in C(\mathbb{B}_n)\cap L^{\infty}: f\text{ vanishes on }\mathbb{B}_n\backslash E\}$ where $\nu(E)$ can be as small as we please. We can also take $G=\{f\in L^{\infty}: f\text{ vanishes on }\mathbb{B}_n\backslash E\}$ where $E$ is a closed nowhere dense subset of $\mathbb{B}_n$ with $\nu(E)$ as small as we please. From these results, one may ask the question: does $\mathfrak{T}$ contain any non-trivial ideal besides $\mathcal{K}$? The purpose of this paper is to show that $\mathfrak{T}$ does contain other ideals despite the above ``extremal cases''. The main result of the paper is the following theorem.

\begin{theorem}[Main Theorem] To every closed subset $F$ of $\mathbb{S}_n$ corresponds a subset $G_{F}$ of $L^{\infty}$ so that the followings hold true:
\begin{enumerate}
     \item $\mathfrak{I}(G_{\emptyset})=\mathcal{K}$ and $\mathfrak{CI}(G_{\mathbb{S}_n}\cap C(\mathbb{B}_n))=\mathfrak{T}.$
     \item If $F_1, F_2$ are closed subsets of $\mathbb{S}_n$ and $F_1\subset F_2$ then $G_{F_1}\subset G_{F_2}.$
     \item If $F_1, F_2$ are closed subsets of $\mathbb{S}_n$ and $F_2\backslash F_1\neq\emptyset$ then $\mathfrak{CI}(G_{F_2}\cap C(\mathbb{B}_n))\backslash\mathfrak{I}(G_{F_1})\neq\emptyset.$ In particular, if $\emptyset\neq F\subsetneq\mathbb{S}_n$ then $\mathcal{K}\subsetneq\mathfrak{CI}(G_F\cap C(\mathbb{B}_n))\subseteq\mathfrak{I}(G_F)\subsetneq\mathfrak{T}.$
\end{enumerate}
\end{theorem}

In the rest of the paper we will provide some preliminaries and basic results in Section \ref{section-2} and Section \ref{section-3}. In Section \ref{section-4}, we give the proof for our Main Theorem.

\section{PRELIMINARIES}\label{section-2}
For any $z\in \mathbb{B}_n,$ let $\varphi_{z}$ denote the Mobius automorphism of $\mathbb{B}_n$ that interchanges $0$ and $z.$ For any $z, w\in \mathbb{B}_n,$ let $\rho(z,w)=|\varphi_{z}(w)|.$ Then $\rho$ is a metric on $\mathbb{B}_n$ (called the pseudo-hyperbolic metric) which is invariant under the action of the group of automorphisms Aut($\mathbb{B}_n$) of $\mathbb{B}_n$. For any $w\in\mathbb{B}_n, \rho(z,w)\rightarrow 1$ as $|z|\rightarrow 1.$ These properties and the following inequality can be proved by using the identities in \cite[Theorem 2.2.2]{Rudin1980}.

For any $z,w,u\in\mathbb{B}_n,$ we have
\begin{equation}\label{eqn-4}
\rho(z,w) \leq\dfrac{\rho(z,u)+\rho(u,w)}{1+\rho(z,u)\rho(u,w)}.
\end{equation}

For any $a$ in $\mathbb{B}_n$ and any $0<r<1,$ define
\begin{equation*}
B(a,r) = \{z\in\mathbb{B}_n: |z-a|<r\},
\end{equation*}
and
\begin{equation*}
E(a,r) = \{z\in\mathbb{B}_n: \rho(z,a)<r\}.
\end{equation*}

Inequality \eqref{eqn-4} shows that if $z,w\in\mathbb{B}_n$ so that $E(z,r_1)\cap E(w,r_2)\neq\emptyset$ for some $0<r_1, r_2<1$ then
\begin{align*}
\rho(z,w)
& \leq\dfrac{\rho(z,u)+\rho(u,w)}{1+\rho(z,u)\rho(u,w)} \text{ (here $u$ is any element in $E(z,r_1)\cap E(w,r_2)$)}\\
& <\dfrac{r_1+r_2}{1+r_1r_2}.
\end{align*}

This implies that if $\rho(z,w)\geq\dfrac{r_1+r_2}{1+r_1r_2}$ then $E(z,r_1)\cap E(w,r_2)=\emptyset.$

\begin{lemma}\label{lemma-4} For any $0<r<1$ and $\zeta\in\mathbb{S}_n$ there is an increasing sequence $\{t_m\}_{m=1}^{\infty}\subset (0,1)$ so that $t_m\rightarrow 1$ as $m\rightarrow\infty$ and $E(t_k\zeta, r)\cap E(t_l\zeta, r)=\emptyset$ for all $k\neq l.$
\end{lemma}

\begin{proof}
We will construct the sequence $\{t_m\}_{m=1}^{\infty}$ by induction. We begin by taking any $t_1$ in $(0,1).$ Suppose we have chosen $t_1<\cdots<t_m$ so that $1-\dfrac{1}{j}<t_j<1$ for all $1\leq j\leq m$ and $E(t_k\zeta, r)\cap E(t_l\zeta, r)=\emptyset$ for all $1\leq k<l\leq m,$ where $m\geq 1.$ Since $\rho(t\zeta,t_j)\rightarrow 1$ as $t\uparrow 1$ for all $1\leq j\leq m,$ we can choose $t_{m+1}\in (0,1)$ with $\max\{t_m, 1-\dfrac{1}{m+1}\}<t_{m+1}$ so that $\rho(t_{m+1}\zeta, t_j)>\dfrac{2r}{1+r^2}$ for all $1\leq j\leq m.$ It then follows that $E(t_{m+1}\zeta,r)\cap E(t_{j}\zeta,r)=\emptyset$ for all $1\leq j\leq m.$ Also, since $1-\dfrac{1}{m}<t_m<1$ for all $m,\ t_m\rightarrow 1$ as $m\rightarrow\infty.$
\end{proof}

\begin{lemma}\label{lemma-2} For any $0<r<1$ and any $\epsilon>0,$ there is a $\delta$ depending on $r$ and $\epsilon$ so that for all $\zeta\in\mathbb{S}_n$ and all $a\in\mathbb{B}_n$ with $|a-\zeta|<\delta,$
\begin{equation*}
E(a,r)\subset\{z\in\mathbb{B}_n: |z-\zeta|<\epsilon\}.
\end{equation*}

As a consequence, if $b\in\mathbb{B}_n$and $\zeta\in\mathbb{B}_n$ so that $E(b,r)\cap\{z\in\mathbb{B}_n: |z-\zeta|<\delta\}\neq\emptyset$ then $|b-\zeta|<\epsilon.$
\end{lemma}

\begin{proof}
From \cite[Section 2.2.7]{Rudin1980}, for any $a\neq 0,$
\begin{equation*}
E(a,r) = \{z\in\mathbb{B}_n: \dfrac{|P z-c|^2}{r^{2}s^2}+\dfrac{|Q z|^2}{r^2s}<1\},
\end{equation*}
where $Pz=\dfrac{\langle z,a\rangle}{\langle a, a\rangle}a,\ Qz = z-Pz,\ c=\dfrac{1-r^2}{1-r^2|a|^2}a$ and $s=\dfrac{1-|a|^2}{1-r^2|a|^2}.$

It then follows that $E(a,r)\subset B(c, r\sqrt{s}).$

Since $|a-c|=r^2s|a|\leq r\sqrt{s},$ we get $B(c,r\sqrt{s})\subset B(a,2r\sqrt{s}).$ Hence $E(a,r)\subset B(a,2r\sqrt{s}).$ Note that the inclusion certainly holds for $a=0$ (in this case $s=1$).

Now suppose $|a-\zeta|<\delta.$ Then $|a|\geq |\zeta|-|a-\zeta|> 1-\delta.$ Hence,
\begin{equation*}
s = \dfrac{1-|a|^2}{1-r^2|a|^2} \leq \dfrac{1-|a|^2}{1-r^2} \leq\dfrac{2(1-|a|)}{1-r^2}<\dfrac{2\delta}{1-r^2}.
\end{equation*}

So for any $z\in E(a,r),$
\begin{equation*}
|z-\zeta| \leq |z-a| + |a-\zeta| \leq 2r\sqrt{s} + \delta < 2r\sqrt{\dfrac{2\delta}{1-r^2}}+\delta.
\end{equation*}

Choosing $\delta$ so that $2r\sqrt{\dfrac{2\delta}{1-r^2}}+\delta<\epsilon,$ we then have the first conclusion of the lemma.

Now suppose $a, b\in\mathbb{B}_n$ so that $|a-\zeta|<\delta$ and $a\in E(b,r).$ Then since $b\in E(a,r),$ the first conclusion of the lemma implies $|b-\zeta|<\epsilon.$
\end{proof}

For any $z\in B_n,$ the formula
\begin{equation*}
U_{z}(f) = (f\circ\varphi_z)k_z,\quad f\in L^2
\end{equation*}
defines a bounded operator on $L^2.$ It is well-known that $U_z$ is a unitary operator with $L^2_a$ as an invariant subspace and $U_{z}T_{f}U_{z}^{*}=T_{f\circ\varphi_{z}}$ for all $z\in B_n$ and all $f\in L^{\infty}.$ See, for example, \cite[Lemmas 7 and 8]{Lee}.

\begin{lemma}\label{lemma-1} For any sequence $\{z_m\}_{m=1}^{\infty}\subset\mathbb{B}_{n}$ with $|z_m|\rightarrow 1,\ U_{z_m}\rightarrow 0$ in the weak operator topology of $\mathfrak{B}(L^2_a).$
\end{lemma}

\begin{proof}
Since span$(\{k_z: z\in\mathbb{B}_n\})$ is dense in $L^2_a,$ it suffices to show that $\lim\limits_{m\rightarrow\infty}\langle U_{z_m}k_z, k_w\rangle=0$ for all $z,w$ in $\mathbb{B}_n.$

Fix such $z, w\in\mathbb{B}_n.$ For each $m\geq 1,$
\begin{align*}
\langle U_{z_m}k_z, k_w\rangle & = (1-|w|^2)^{(n+1)/2}(U_{z_m}k_z)(w)\\
& = (1-|w|^2)^{(n+1)/2}k_z(\varphi_{z_m}(w))k_{z_m}(w)\\
& = \dfrac{\bigl((1-|w|^2)(1-|z|^2)(1-|z_m|^2)\bigr)^{(n+1)/2}}{\bigl((1-\langle\varphi_{z_{m}}(w),z\rangle)(1-\langle w, z_m\rangle\bigr)^{n+1}}.
\end{align*}

Since $|\langle\varphi_{z_m}(w),z\rangle|\leq |z|$ and $|\langle w, z_m\rangle|\leq |w|,$  we obtain
\begin{equation*}
|\langle U_{z_m}k_z, k_w\rangle| \leq \dfrac{\bigl((1-|w|^2)(1-|z|^2)(1-|z_m|^2)\bigr)^{(n+1)/2}}{\bigl((1-|z|)(1-|w|)\bigr)^{n+1}}.
\end{equation*}

It then follows that $\lim\limits_{m\rightarrow\infty}\langle U_{z_m}k_z, k_w\rangle = 0.$
\end{proof}

\begin{lemma}\label{lemma-3}
Let $\{z_{m}\}_{m=1}^{\infty}\subset\mathbb{B}_n$ so that $|z_m|\rightarrow 1$ as $m\rightarrow\infty.$ Let $S$ be any non-zero positive operator on $L^2_a.$ Suppose $A=\sum\limits_{m=1}^{\infty}U_{z_m}SU_{z_m}^{*}$ exists in the strong operator topology and is a bounded operator on $L^2_a$. Then there is a constant $c>0$ and an $f\in L^2_a$ so that $\|AU_{z_m}f\|\geq c>0$ for all $m.$
\end{lemma}

\begin{proof}
Since $S$ is non-zero and positive, there is an $f\in L^{2}_a$ with $\|f\|=1$ so that $\langle Sf, f\rangle>0.$ For each $m\geq 1,$
\begin{align*}
\langle AU_{z_m}f, U_{z_m}f\rangle & \geq \langle U_{z_m}SU_{z_m}^{*}U_{z_m}f, U_{z_m}f\rangle\\
& = \langle SU_{z_m}^{*}U_{z_m}f, U_{z_m}^{*}U_{z_m}f\rangle\\
& = \langle Sf, f\rangle.
\end{align*}

Since $\|U_{z_m}f\|=1,$ it follows that $\|AU{z_m}f\|\geq\langle Sf, f\rangle>0$ for all $m\geq 1.$
\end{proof}


\section{BASIC RESULTS}\label{section-3}
The first result in this section shows that for a certain class of subsets $G$ of $L^{\infty},\ \mathfrak{I}(G)$ possesses a special property. This property will later help us distinguish $\mathfrak{I}(G_1)$ and $\mathfrak{I}(G_2)$ when $G_1\neq G_2.$

\begin{proposition}\label{prop-1}
Let $W$ be a subset of $\mathbb{B}_n$ and let $F={\rm cl}(W)\cap\mathbb{S}_{n}.$ Let $f$ be in $L^{\infty}$ so that $f$ vanishes almost everywhere in $\mathbb{B}_n\backslash W.$ Let $g_1,\ldots, g_l$ be any functions in $L^{\infty}.$ Let $\{z_m\}_{m=1}^{\infty}$ be any sequence in $\mathbb{B}_n$ so that $|z_m|\rightarrow 1$ and $|z_m-w|\geq\epsilon>0$ for all $w\in F,$ all $m\geq 1,$ where $\epsilon$ is a fixed constant. Then the sequence $\{T_{f}T_{g_1}\cdots T_{g_l}U_{z_m}\}_{m=1}^{\infty}$ converges to $0$ in the strong operator topology of $\mathfrak{B}(L^2_a).$ Consequently, if we set
\begin{equation*}
G=\{f\in L^{\infty}: f\text{ vanishes almost everywhere in }\mathbb{B}\backslash W\}
\end{equation*}
then for any $T\in\mathfrak{I}(G),\ TU_{z_m}\rightarrow 0$ in the strong operator topology of $\mathfrak{B}(L^2_a).$
\end{proposition}

\begin{proof}
Let $V_1=\{|z|\leq 1: |z-w|<\epsilon/3\text{ for some }w\in F\}$ and $V_2=\{|z|\leq 1: |z-w|<\epsilon/2\text{ for some }w\in F\}.$ Let $\eta$ be a continuous function on $\overline{\mathbb{B}}_n$ so that $0\leq\eta\leq 1,\ \eta(z)=1$ if $z\in {\rm cl}(V_1)$ and $\eta(z)=0$ if $z\notin V_2.$ Let $Z={\rm cl}(W)\cap(\overline{\mathbb{B}}_n\backslash V_1).$ Then $Z\subset\overline{\mathbb{B}}_n$ and $Z$ is compact with respect to the Euclidean metric. We have
\begin{equation*}
Z\cap\mathbb{S}_{n} = {\rm cl}(W)\cap\mathbb{S}_n\cap(\overline{\mathbb{B}}_n\backslash V_1) = F\cap(\overline{\mathbb{B}}_n\backslash V_1) = \emptyset.
\end{equation*}
Thus $Z$ is a compact subset of $\mathbb{B}_{n}.$ Since the function $f(1-\eta)$ vanishes almost everywhere in $(\mathbb{B}_n\backslash W)\cup {\rm cl}(V_1)$ which contains $\mathbb{B}_n\backslash Z,$ the operator $T_{f(1-\eta)}$ is compact.

Since $T_{g}T_{1-\eta}-T_{g(1-\eta)}$ and $T_{1-\eta}T_{g}-T_{g(1-\eta)}$ are compact for all $g\in L^{\infty}$ (see \cite{Coburn}), we have
\begin{align}\label{eqn-1}
T_{f}T_{g_1}\cdots T_{g_l}
& = T_{f}T_{g_1}\cdots T_{g_l}T_{\eta} + T_{f}T_{g_1}\cdots T_{g_l}T_{1-\eta}\notag\\
& = T_{f}T_{g_1}\cdots T_{g_l}T_{\eta} + T_{f(1-\eta)}T_{g_1}\cdots T_{g_l} + K_1\\
& = T_{f}T_{g_1}\cdots T_{g_l}T_{\eta} + K,\notag
\end{align}
where $K_1$ is a compact operator and $K = T_{f(1-\eta)}T_{g_1}\cdots T_{g_l} + K_1$ is also a compact operator.

For any $h\in L^2_a\cap L^{\infty}$ and any $m\geq 1$ we have
\begin{align*}
\|T_{\eta}U_{z_m}h\|^2 & \leq \|\eta U_{z_m}h\|^2\\
& \leq \int\limits_{V_2}|(U_{z_m}h)(z)|^{2}\ {\rm d}\nu(z)\\
& = \int\limits_{V_2}|h(\varphi_{z_m}(z))k_{z_m}(z)|^2\ {\rm d}\nu(z)\\
& \leq \|h\|^2_{\infty}\int\limits_{V_2}|k_{z_m}(z)|^2\ {\rm d}\nu(z).
\end{align*}

Let $V_3=\{|z|\leq 1: |z-w|<\epsilon\text{ for some }w\in F\}.$ Since the map $(z,w)\mapsto |1-\langle z,w\rangle|$ is continuous and does not vanish on the compact set ${\rm cl}(V_2)\times (\overline{\mathbb{B}}_n\backslash V_3),$ there is a $\delta>0$ so that $|1-\langle z,w\rangle|\geq\delta$ for all $z\in {\rm cl}(V_2)$ and $w\in (\overline{\mathbb{B}}_n\backslash V_3).$

For each $m\geq 1,\ z_m\in (\mathbb{B}_n\backslash V_3),$ so for all $z\in V_2,$
\begin{equation*}
|k_{z_m}(z)| = \dfrac{(1-|z_m|^2)^{(n+1)/2}}{|1-\langle z, z_m\rangle|^{n+1}} \leq \dfrac{(1-|z_m|^2)^{(n+1)/2}}{\delta^{n+1}}.
\end{equation*}

Hence we have
\begin{equation*}
\|T_{\eta}U_{z_m}h\|\leq \|h\|_{\infty}\sqrt{\nu(V_2)}\dfrac{(1-|z_m|^2)^{(n+1)/2}}{\delta^{n+1}}.
\end{equation*}

This implies $\|T_{\eta}U_{z_m}h\|\rightarrow 0$ as $m\rightarrow\infty.$ Since $L^2_a\cap L^{\infty}$ is dense in $L^2_a$ and $\|T_{\eta}U_{z_m}\|\leq\|T_{\eta}\|\leq 1$ for all $m,$ we conclude that $T_{\eta}U_{z_m}\rightarrow 0$ in the strong operator topology of $\mathcal{B}(L^2_a).$ So $T_{f}T_{g_1}\cdots T_{g_l}T_{\eta}U_{z_m}\rightarrow 0$ in the strong operator topology of $\mathfrak{B}(L^2_a).$ Also by Lemma \ref{lemma-1}, $U_{z_m}\rightarrow 0$ in the weak operator topology, so $KU_{z_m}\rightarrow 0$ in the strong operator topology for any compact operator $K.$ Combining these facts with \eqref{eqn-1}, we conclude that $T_{f}T_{g_1}\cdots T_{g_l}U_{z_m}\rightarrow 0$ in the strong operator topology of $\mathfrak{B}(L^2_a).$
\end{proof}


The following proposition was proved by Su{\'a}rez for the case $n=1$ (see \cite[Proposition 2.9]{Suarez2004}). The case $n\geq 2$ is similar and can be proved with the same method. The point is that for $n\geq 2,$ the metric $\rho$ and the reproducing kernel functions have all the properties needed for Su{\'a}rez's proof. See \cite[Section 2]{Le2006} for more details.

\begin{proposition}\label{prop-4} Let $0<r<1$ and $\{w_{m}\}_{m=1}^{\infty}$ be a sequence in $\mathbb{B}_n$ so that $E(w_{k},r)\cap E(w_{l},r)=\emptyset$ for all $k\neq l.$ For each $m\in\mathbb{N}$, let $c_{m}^{1}, \ldots, c_{m}^{l},$ $a_m, b_{m},$ $d_{m}^{1}, \ldots, d_{m}^{k}\in L^{\infty}$ be functions of norm $\leq 1$ that vanish almost everywhere on $\mathbb{B}_n\backslash E(w_{m},r).$ Then
\begin{equation*}
\sum\limits_{m\in\mathbb{N}}T_{c_{m}^1}\ldots T_{c_{m}^{l}}(T_{a_m}T_{b_m}-T_{b_m}T_{a_m})T_{d_{m}^{1}}\cdots T_{d_{m}^{k}}
\end{equation*}
belongs to $\mathfrak{CI}(L^{\infty})$.
\end{proposition}

\begin{remark}\label{remark-1} In the proof of Proposition \ref{prop-4}, we work only with Toeplitz operators with symbols in the set $G$ which consists of functions of the form $\sum\limits_{m\in F}f_{m},$ where $F$ is a subset of $\mathbb{N}$ and $f$ is one of the symbols $c^{1},\ldots, c^{l},$ $a, b,$ $d^{1}, \ldots, d^{k}.$ So in the conclusion of the proposition, we may replace $\mathfrak{CI}(L^{\infty})$ by $\mathfrak{CI}(G).$
\end{remark}

\section{PROOF OF THE MAIN THEOREM}\label{section-4}
Now we are ready for the proof of our main result.

Fix $0<r<1.$ Let $W_{\emptyset}=E(0,r).$ For any closed non-empty subset $F$ of $\mathbb{S}_n,$ let
\begin{equation*}
W_{F} = \bigcup\limits_{0<t<1}\bigcup\limits_{\zeta\in F}E(t\zeta, r).
\end{equation*}

It is clear that $W_{\mathbb{S}_n}=\mathbb{B}_n.$ We always have $F\subset{\rm cl}(W_{F})\cap\mathbb{S}_n$. We will show that in fact $F = {\rm cl}(W_{F})\cap \mathbb{S}_n.$ Suppose $\zeta\in {\rm cl}(W_{F})\cap\mathbb{S}_n.$ For any $m\geq 1,$ applying Lemma \ref{lemma-2} with $\epsilon=\dfrac{1}{m},$ we get a $\delta_m>0$ so that if $b\in\mathbb{B}_n$ and $E(b,r)\cap V_{m}\neq\emptyset$, where $V_{m}=\{z\in\mathbb{B}_n: |z-\zeta|<\delta_m\},$ then $|b-\zeta|<\dfrac{1}{m}.$ Since $W_{F}\cap V_{m}\neq\emptyset,$ there is $t_m\in (0,10)$ and $\zeta_m\in F$ so that $E(t_m\zeta_m, r)\cap V_{m}\neq\emptyset.$ Hence $|t_m\zeta_m - \zeta|<\dfrac{1}{m}.$ So $|t_m\zeta_m-\zeta|\rightarrow 0$ as $m\rightarrow\infty.$ Since $t_m = |t_m\zeta_m|\rightarrow |\zeta|=1$ as $m\rightarrow\infty,$ we get $\zeta_m = t^{-1}_{m}(t_m\zeta_m)\rightarrow \zeta$ in the Euclidean metric as $m\rightarrow\infty.$ This implies that $\zeta\in F.$ Thus, ${\rm cl}(W_{F})\cap\mathbb{S}_n \subset F.$

Now define
\begin{equation*}
G_{F} = \{f\in L^{\infty}: f\text{ vanishes almost everywhere in }\mathbb{B}_n\backslash W_{F}\}.
\end{equation*}

It is clear that if $F_1\subset F_2$ then $W_{F_1}\subset W_{F_2},$ hence $G_{F_1}\subset G_{F_2}.$

Since $T_f$ is compact for all $f\in G_{\emptyset},\ \mathfrak{I}(G_{\emptyset})=\mathcal{K}.$ Since $G_{\mathbb{S}_n}=L^{\infty},$ we have $\mathfrak{CI}(G_{\mathbb{S}_n}\cap\ C(\mathbb{B}_n))=\mathfrak{CI}(L^{\infty}\cap\mathbb{B}_n))=\mathfrak{T}.$

Now suppose $F_1$ and $F_2$ are closed subsets of $\mathbb{S}_n$ so that $F_2\backslash F_1\neq\emptyset.$ Let $\zeta\in F_2$ so that $\zeta\notin F_1.$ From Lemma \ref{lemma-4} there is a sequence $\{t_m\}_{m=1}^{\infty}\subset (0,1)$ with $t_m\uparrow 1$ and $E(t_k\zeta, r)\cap E(t_l\zeta, r)=\emptyset$ for all $k\neq l.$ Let $z_m=t_m\zeta$ for all $m\geq 1.$ Since $|z_m-\zeta|\rightarrow 0$ and $\zeta\notin F_1$ which is a closed subset of $\mathbb{S}_n$, there is an $\epsilon>0$ so that $|z_m-w|\geq\epsilon$ for all $w\in F_1,$ all $m\geq 1.$ Since ${\rm cl}(W_{F_1})\cap\mathbb{S}_n = F_1,$ Proposition \ref{prop-1} shows that $TU_{z_m}\rightarrow 0$ in the strong operator topology for all $T\in\mathfrak{I}(G_{F_1}).$

Take $f$ to be any continuous function supported in $E(0,r)$ such that $[T_f, T_{\overline{f}}]\neq 0.$ Any function of the form $f(z)=z_1\eta(|z|/r)$ where $\eta$ is non-negative, continuous and supported in $[0,1]$ with $\|\eta\|_{\infty}>0$ will work. Let $S=[T_{f}, T_{\overline{f}}]^2$ then $S$ is a non-zero, positive operator on $L^2_a.$ Define
\begin{equation*}
T = \sum\limits_{m=1}^{\infty} U_{z_m}SU^{*}_{z_m}.
\end{equation*}

By Lemma \ref{lemma-3}, there is a constant $c>0$ and an $h\in L^2_a$ so that $\|TU_{z_m}h\|\geq c$ for all $m.$ This implies that $T\notin\mathfrak{I}(G_{F_1}).$

For each $m,$
\begin{align*}
U_{z_m}[T_{f}, T_{\overline{f}}]U^{*}_{z_m}
& = U_{z_m}T_{f}T_{\overline{f}}U^{*}_{z_m} - U_{z_m}T_{\overline{f}}T_{f}U^*_{z_m}\\
& = T_{f\circ\varphi_{z_m}}T_{\overline{f}\circ\varphi_{z_m}} - T_{\overline{f}\circ\varphi_{z_m}}T_{f\circ\varphi_{z_m}}\\
& = [T_{f\circ\varphi_{z_m}}, T_{\overline{f}\circ\varphi_{z_m}}].
\end{align*}

So $U_{z_m}SU^{*}_{z_m} = [T_{f\circ\varphi_{z_m}}, T_{\overline{f}\circ\varphi_{z_m}}]^2.$ Hence $T = \sum\limits_{m=1}^{\infty}[T_{f\circ\varphi_{z_m}}, T_{\overline{f}\circ\varphi_{z_m}}]^2.$

Since each $f\circ\varphi_{z_m}$ is supported in $\{w\in\mathbb{B}_n: |\varphi_{z_m}(w)|<r\}=E(z_{m},r)$ and $E(z_k,r)\cap E(z_l,r)=\emptyset$ for all $k\neq l$, Proposition \ref{prop-4} and Remark \ref{remark-1} show that $T\in\mathfrak{CI}(G_{F_2}\cap C(\mathbb{B}_n)).$ So $T\in\mathfrak{CI}(G_{F_2})\backslash\mathfrak{I}(G_{F_1}).$

\end{document}